\documentclass[reqno, oneside, 12pt]{amsart}

\usepackage[letterpaper]{geometry}
\usepackage{url, amssymb,amsmath,amsthm,amsxtra,mathtools,mathrsfs,calc,nccmath,bm,color,enumitem}
\usepackage{needspace}
\mathtoolsset{showonlyrefs}

\def\Z{\mathbb{Z}}
\def\Q{\mathbb{Q}}

\def\C{\mathbb{C}}
\def\({\left(}
\def\){\right)}
\newcommand{\ord}{\operatorname{ord}}
\newcommand{\SL}{\operatorname{SL}}
\newcommand{\GL}{\operatorname{GL}}
\newcommand{\Frob}{\operatorname{Frob}}
\newcommand{\Sh}{\operatorname{Sh}}
\newcommand{\Gal}{\operatorname{Gal}}

\newcommand{\pfrac}[2]{\left(\frac{#1}{#2}\right)}
\newcommand{\pmfrac}[2]{\left(\frac{#1}{#2}\right)}

\newcommand{\pMatrix}[4]{\left(\begin{matrix}#1 & #2 \\ #3 & #4\end{matrix}\right)}
\newcommand{\ppmatrix}[4]{\left(\begin{smallmatrix}#1 & #2 \\ #3 & #4\end{smallmatrix}\right)}
\newcommand{\sk}{\big|_k }

\renewcommand{\(}{\left(}
\renewcommand{\)}{\right)}
\usepackage{calc}
\usepackage{graphicx}
\usepackage{caption}
\usepackage{subcaption}

\def\ep{\ep}

\def\ep{\varepsilon}

\let\temp\phi
\let\phi\varphi
\let\varphi\temp
\newcommand{\spmod}[1]{\ensuremath{\,(#1)}}
\newcommand{\Qbar}{\overline{\Q}}
\newcommand{\F}{\mathbb{F}}
\newcommand{\Fbar}{\overline{\F}}
\newcommand{\cO}{\mathcal{O}}
\newcommand{\rhobar}{\overline{\rho}}
\DeclareMathOperator{\PSL}{PSL}
\DeclareMathOperator{\PGL}{PGL}

\newtheorem{theorem}{Theorem}[section]
\newtheorem{lemma}[theorem]{Lemma}

\newtheorem{proposition}[theorem]{Proposition}

\theoremstyle{remark}
\newtheorem*{remark}{Remark}

\newtheorem*{notation}{Notation}
\numberwithin{equation}{section}

\AtBeginDocument{%
   \def\MR#1{}
}

\begin{document}


\title{Congruences like Atkin's for the partition function}
\author{Scott Ahlgren}
\address{Department of Mathematics\\
University of Illinois\\
Urbana, IL 61801} 
\email{sahlgren@illinois.edu} 
\author{Patrick B. Allen}
\address{Department of Mathematics and Statistics\\
McGill University\\ 
Montreal, Quebec H3A 0B9
} 
\email{patrick.allen@mcgill.ca} 
\author{Shiang Tang}
\address{Department of Mathematics\\
Purdue University\\
West Lafayette, IN 47907}
\email{tang573@purdue.edu} 

\thanks{The first author was  supported by a grant from the Simons Foundation (\#426145 to Scott Ahlgren).
The second author was supported by grants from the NSF (DMS-1902155) and NSERC.
The third author was supported by a grant from the NSF (DMS-1902155).
}

 
\begin{abstract}
Let $p(n)$ be the ordinary partition function.  In the 1960s Atkin found a number of examples of congruences
of the form $p( Q^3 \ell n+\beta)\equiv0\pmod\ell$ where  $\ell$ and $Q$ are prime and  $5\leq \ell\leq 31$; these lie in two natural families distinguished by the square class of $1-24\beta\pmod\ell$.
 In recent decades much work has been done to understand congruences of the form 
 $p(Q^m\ell n+\beta)\equiv 0\pmod\ell$.  It is now known that there are many such congruences when $m\geq 4$, that such congruences are scarce (if they exist at all) when $m=1, 2$, and
that for $m=0$ such congruences exist only when $\ell=5, 7, 11$.   For congruences like Atkin's (when $m=3$), more examples have been found for $5\leq \ell\leq 31$ but little else seems to be known.

Here we use the theory of modular Galois representations to prove that for every prime $\ell\geq 5$, there are infinitely many congruences like Atkin's 
in the first natural family which he discovered and  that for at least $17/24$ of the primes $\ell$ there are infinitely many congruences in the second family.
\end{abstract}

\maketitle


\section{Introduction}
The partition function $p(n)$ counts the number of ways to write the positive integer $n$ as the sum of a non-increasing sequence of positive integers
(by convention we agree that $p(0)=1$ and that $p(n)=0$ if $n\not\in\{0, 1, 2, \dots\}$).
The study of the arithmetic properties of $p(n)$ has a long and rich history;  interest in this topic stems not only from the fact that $p(n)$ is a fundamental function in additive number theory and combinatorics, but also from the fact that its generating function is a modular form of weight $-\frac12$ on the full modular group.  

The most famous examples of arithmetic phenomena for the partition function are the  Ramanujan congruences
\begin{gather}\label{eq:ramcong}
p(\ell n+\beta_\ell)\equiv 0\pmod\ell
\quad\text{for }\ell=5, 7, 11
\text{,}
\end{gather}
where $\beta_\ell:=\frac1{24}\pmod\ell$.
Extensions of these results for arbitrary powers of $5, 7, 11$ were conjectured and proved by Ramanujan, Watson and Atkin
\cite{ram_someprop, ram_cong_PLMS,  ram_cong_mathz, watson,atkin_glasgow}.
On the other hand, after the work of the first author and Boylan \cite{Ahlgren-Boylan} it is known that there are no congruences of the form \eqref{eq:ramcong} 
with $\ell\geq 13$.

Further examples of congruences for primes $\ell\leq 31$ were found by Newman, Atkin,  and O'Brien \cite{newman, atkin-obrien, atkin_mult}.
These examples  take the form
\begin{gather}\label{eq:partcong}
p( Q^m \ell n+\beta)\equiv 0\pmod\ell
\text{,}
\end{gather}
where $Q$ is a prime distinct from $\ell$, and $m=3$ or $4$.

Many years later, Ono \cite{Ono_annals} showed that for every $\ell\geq 5$, there are infinitely many primes $Q$ for which we have a congruence \eqref{eq:partcong} with $m=4$.  After the work of the first author and Ono \cite{Ahlgren-Ono}, we have the following (see \cite[Thm. 1]{ahlgren-ono_conj}).

\begin{theorem} Suppose that  $\ell\geq 5$ is prime and that $\pfrac{1-24\beta}\ell\in \{0, - 1\}$.
 Then a positive proportion of primes $Q\equiv -1\pmod\ell$ have the property that 
\begin{gather}\label{eq:aothm}
p\pmfrac{Q^3n+1}{24}\equiv 0\pmod \ell\qquad \text{if $Q\nmid n$  \ and  \ $n\equiv 1-24\beta\pmod\ell$.}
\end{gather}
\end{theorem}
For any such $\beta$, selecting $n$ in one of $Q-1$ residue classes modulo $Q$  gives a congruence of the form
\begin{gather}\label{eq:fourthpowercong}
p(Q^4\ell n+\beta')\equiv 0\pmod\ell
\end{gather}
with  $\pfrac{1-24\beta'}\ell=\pfrac{1-24\beta}\ell$.
Radu \cite{Radu_aoconj} confirmed a conjecture of the first author and Ono by proving that if there is a congruence
\begin{gather*}
p(mn+\beta)\equiv 0\pmod\ell
\quad\text{with }\ell\geq 5\ \ \text{prime}
\text{,}
\end{gather*}
then $\ell\mid m$ and $\pfrac{1-24\beta}\ell\in \{0, -1\}$.

After this discussion we know that there are many congruences of the form  \eqref{eq:partcong} with $m\geq 4$ and  no congruences other than 
\eqref{eq:ramcong} with $m=0$.  It  therefore  becomes natural to ask about the existence of such congruences  when 
$m=1, 2, 3$.

Recent work of the first author, Beckwith and Raum \cite{ahlgren-beckwith-raum} has shown that for  $m=1$ and $m=2$, and for any prime $\ell\geq 5$,  congruences of this form
(if they exist at all) are extremely scarce in a precise sense.    Since the main theorems of that paper require some notation to state, we  mention here only 
Corollary 1.2:
If $17\leq \ell<10000$ is prime, and $S$ is the set of primes $Q$ for which there is a congruence
\begin{gather*}
p(Q\ell n +\beta)\equiv 0\pmod\ell,
\end{gather*}
then $S$ has density zero.

This leaves open only the case $m=3$, which is the focus of this paper.
In this case, Atkin \cite{atkin_mult} discovered many congruences of the form
\begin{gather}\label{eq:atkincong}
p\(Q^3 \ell  n+\beta\)\equiv 0\pmod \ell
\end{gather}
 for small primes $\ell$.  These arise from two families   which we describe in detail.
 
Let $13\leq \ell\leq 31$ be prime. 
 Atkin   \cite[eq. (52)]{atkin_mult} gave examples of primes $Q$ such that
 \begin{gather}\label{eq:atkinI}
 p\pmfrac{Q^2\ell n+1}{24}\equiv 0\pmod\ell \ \ \text{if}\ \ \pmfrac{n}{Q}=\ep_Q
 \end{gather}
 for some $\ep_Q\in\{\pm1\}$.
 Fixing $n$ in one of the allowable residue classes modulo $Q$ produces a congruence of the form \eqref{eq:atkincong}.
For these small values of $\ell$, the relevant generating functions are   eigenforms of the Hecke operators, 
and Atkin's method relies on finding what he calls ``accidental" eigenvalues (Atkin works with modular functions rather than modular forms, but the effect is the same).
 We will say that congruences \eqref{eq:atkinI} are of type ``Atkin I."

 Later, Weaver \cite{weaver} found  more accidental eigenvalues for these primes (as well as  more examples of congruences \eqref{eq:aothm} with $\ell\mid n$). 
As an application of his performant algorithm to compute large values of the partition function,  Johansson  \cite{johansson}
extended this list substantially; there are now more than $22$ billion examples for primes $\ell\leq 31$.

For each of  $\ell=5, 7$ and $13$, Atkin \cite[Thm. 1, 2]{atkin_mult} showed that if $Q\equiv -2\pmod\ell$ then 
\begin{gather}\label{eq:atkinII}
 p\pmfrac{Q^2 n+1}{24}\equiv 0\pmod\ell \ \ \text{if}\ \ \pmfrac{-n}\ell=-1\ \ \text{and} \ \ \pmfrac{-n}{Q}=-1.
 \end{gather}
 We will say that congruences \eqref{eq:atkinII} are of type ``Atkin II."  To the authors' knowledge, no examples of such congruences are known for $\ell\geq 13$. After this discussion there are two natural questions:
 
 \begin{enumerate}
 \item Are there congruences of type Atkin I for  primes $\ell\geq 31$?
 \item Are there  congruences of type Atkin II for  primes $\ell\geq 13$?
 \end{enumerate}
We will refer to these congruences simply as ``Type I" and ``Type II" in what follows.

\begin{remark}
Once $\ell\geq 37$ the spaces of modular forms which are relevant for congruences of Type  I  are no longer one-dimensional.  However, one may still perform a search for accidents in the sense of Atkin.
For example, when $\ell=37$, the relevant space is two-dimensional.  
A computation of  the Hecke eigenvalues of the two newforms in this space  for $Q<10000$ yields three  ``accidents":
there are three primes $Q$ for which the $Q$th eigenvalue of each  newform lies in the required residue class modulo a prime above $\ell$ in the field generated by its coefficients.
In particular we have a congruence \eqref{eq:atkinI} when $Q=6599$, $7541$, and $9547$. For example, 
\begin{gather*}
 p\pmfrac{6599^2\cdot 37 n+1}{24}\equiv 0\pmod{37} \ \ \ \ \text{if}\ \  \ \  \pmfrac{n}{6599}=-1,
 \end{gather*}
 which leads to $3299$ congruences modulo $37$ of the form \eqref{eq:atkincong} with $m=3$.
 Similarly, we have 
 \begin{gather*}
 p\pmfrac{7541^2\cdot 37 n+1}{24}\equiv 0\pmod{37} \ \ \ \ \text{if}\ \  \ \  \pmfrac{n}{7541}=1.
 \end{gather*}

\end{remark}

Our goal in this paper is to prove that there are many congruences of the types which Atkin discovered.  In particular we will prove the following theorems.
(Note that for $\ell=5, 7, 11$, the statements about  congruences of Type I are trivially true in view of  \eqref{eq:ramcong}.)

The first result shows that congruences of Type I hold for every prime $\ell$
(an explicit description of what is meant by ``positive proportion" is given at the end of the Introduction).
\begin{theorem}\label{thm:Atkin1}   Suppose that $\ell\geq 5$ is prime. 
Then a positive proportion of the  primes $Q\equiv 1\pmod\ell$ have the property that 
\begin{gather*}p\pmfrac{Q^2\ell n+1}{24}\equiv 0\pmod\ell\qquad\text{if}\qquad \pmfrac nQ=\pmfrac{-1}Q^\frac{\ell-3}2.
\end{gather*}
\end{theorem}

The second result shows that for many primes $\ell$ we have  congruences of Types I and II involving primes $Q$ in a different residue class modulo $\ell$.
\begin{theorem}\label{thm:Atkin2}  Suppose that $\ell\geq 5$ is prime, and  that 
\begin{gather}\label{eq:2a}
\text{there exists an integer $a$ with $2^a\equiv -1\pmod\ell$.}
\end{gather}
Then  
\begin{enumerate}
\item 
A positive proportion of primes $Q\equiv -2\pmod\ell$ have the property that for some $\ep_Q\in \{\pm1\}$, we have
\begin{gather}\label{eq:type1}
 p\pmfrac{Q^2\ell n+1}{24}\equiv 0\pmod\ell \ \ \text{if}\ \ \pmfrac{n}{Q}=\ep_Q.
 \end{gather}
\item A positive proportion of primes $Q\equiv -2\pmod\ell$ have the property that 
\begin{gather}
 p\pmfrac{Q^2 n+1}{24}\equiv 0\pmod\ell \ \ \text{if}\ \ \pmfrac{-n}\ell=-1\ \ \text{and} \ \ \pmfrac{-n}{Q}=-1.
 \end{gather}
 \end{enumerate}
\end{theorem}
 \begin{remark} By a result of Hasse \cite{hasse:1966}, the proportion of primes $\ell$ for which \eqref{eq:2a} is satisfied  is $17/24\approx .708$.
 \end{remark}

Finally, we prove an analogous result under a similar assumption at the prime $3$, although the situation here is slightly more complicated.
We use the notation $S_{k}^{\rm new}\(6, \ep_2, \ep_3\)$ to denote the new subspace of modular forms of integral weight $k$ on 
$\Gamma_0(6)$ with  eigenvalues $\ep_2$ and $\ep_3$ under the Atkin-Lehner involutions $W_2$ and $W_3$ (see the next section for details).

\begin{theorem}\label{thm:Atkin3}
Suppose that $\ell\geq 5$ is prime, and that 
\begin{gather}\label{eq:3a}
\text{there exists an integer $a$ with $3^a\equiv -2\pmod\ell$.}
\end{gather}
Suppose further that 
there  is  no congruence  modulo any prime above $\ell$ between distinct newforms in $S_{\ell-3}^{\rm new}\(6, -\pmfrac{8}{-\ell}, -\pmfrac{12}{-\ell}\)$.
Then  
a positive proportion of primes $Q\equiv -2\pmod\ell$ have the property that for some $\ep_Q\in \{\pm1\}$, we have
\begin{gather}\label{eq:type1_3}
 p\pmfrac{Q^2\ell n+1}{24}\equiv 0\pmod\ell \ \ \text{if}\ \ \pmfrac{n}{Q}=\ep_Q.
 \end{gather}

If \eqref{eq:3a} holds and there is no congruence  modulo any prime above $\ell$ between distinct newforms in $S_{\ell^2-3}^{\rm new}\(6, -1, -1\)$
then a positive proportion of primes $Q\equiv -2\pmod\ell$ have the property that 
\begin{gather}\label{eq:type2}
 p\pmfrac{Q^2 n+1}{24}\equiv 0\pmod\ell \ \ \text{if}\ \ \pmfrac{-n}\ell=-1\ \ \text{and} \ \ \pmfrac{-n}{Q}=-1.
 \end{gather}

\end{theorem}

 \begin{remark}
Approximately $82.8\%$ of the primes $\ell<100000$ satisfy either \eqref{eq:2a} or \eqref{eq:3a}.
For those which satisfy $\eqref{eq:3a}$ but not  \eqref{eq:2a} the additional hypothesis that there is no congruence between newforms is required
due to a technical issue which is described in the last section.
One expects that this condition should almost always be satisfied, and a computation shows that there
is  no congruence  modulo any prime above $\ell$ between distinct newforms in $S_{\ell-3}^{\rm new}\(6, -\pmfrac{8}{-\ell}, -\pmfrac{12}{-\ell}\)$ for any 
$\ell<150$.

However it seems very difficult to remove this condition.  For example when $\ell=71$, there is such a congruence between two newforms in the space
$S_{\ell-3}^{\rm new}\(6, \pmfrac{8}{-\ell}, \pmfrac{12}{-\ell}\)$.
In particular, this space contains a Galois orbit  consisting of three newforms  defined over a field which is ramified at $71$,
and this orbit gives rise to only two distinct reductions modulo the prime above $71$.
 \end{remark}

  To prove the existence of congruences \eqref{eq:fourthpowercong}, one must find primes $Q$ for which the Hecke operator of index $Q$
annihilates a suitable space of modular forms modulo some power of $\ell$; the existence of such primes is guaranteed by a result of Serre 
\cite[ex. 6.4]{serre_1976}. This approach extends to congruences modulo powers of $\ell$ \cite{Ahlgren_mathann, Ahlgren-Ono}, and by work of 
Treneer \cite{Treneer_1, Treneer_2} to a wide class of weakly holomorphic modular forms.
The current situation is more delicate; we need to find primes $Q$ for which the
Hecke operator acts diagonally with a specified eigenvalue, which entails a much more careful study of the Galois representations which arise.
In recent work, Raum \cite[Thm. E]{raum_2021} has considered the converse question; given the existence of such a congruence as \eqref{eq:fourthpowercong} or \eqref{eq:atkincong} he deduces information about the generalized Hecke eigenvalues $\lambda_Q$  on a natural subspace of modular forms.

The Galois theoretic results which we prove  may be of independent interest. 
The following theorem is the main theoretical input in the proof of  Theorem~\ref{thm:Atkin1}
 (here $S_k^{\rm new}(6)$ denotes the new subspace of cusp forms on $\Gamma_0(6)$ whose coefficients are integral at all primes above $\ell$).
 \begin{theorem}\label{thm:lowweight}
Suppose that $m$ is a positive integer and that $\ell \ge 5$ is prime.
Then a positive proportion of primes $Q\equiv 1\pmod{\ell^m}$ have the following  property:   for every
$g \in S_{\ell-3}^{\rm new}\(6\)$ we  have 
\begin{gather*}g\big| T(Q)\equiv    g\pmod{\ell^m}.\end{gather*}
\end{theorem}

Above and in what follows, the positive proportion of primes appearing in our theorems are $\{2,3,\ell\}$-frobenian in the sense of \cite[\S 3.3]{serre-NXp} (we do not prove that the set of all primes satisfying the conclusions of our theorems is frobenian, just that it contains a frobenian subset). 
 This means that there is a finite Galois extension $E/\Q$ unramified outside of $\{2,3,\ell\}$ and a subset $C \subseteq \Gal(E/\Q)$ that is a union of conjugacy classes such that the conclusion of the theorem holds for any prime $Q$ with $\Frob_Q \in C$. 
We have not attempted to give a lower bound on the size of $\lvert C \rvert / \lvert G \rvert$, hence on the proportion of such primes $Q$, although in principle this is possible. 
We note the following subtlety.
 We first  prove Theorem~\ref{thm:lowweight}  under the further hypothesis that $g$ is a newform, where the frobenian set is more transparent, but modulo some possibly higher power $\ell^{m'}$ depending on how the lattice spanned by newforms sits inside of $S_{\ell-3}^{\rm new}\(6\)$ (see \S\ref{sec:cong-non-eig}). 
The density of our frobenian set depends on this $m'$, hence on the relationship between these two lattices.

In the next section we give some background on modular forms and Galois representations.  Section~3 is devoted to the proof of Theorem~\ref{thm:lowweight}. The proof is technical, and most of the difficulty arises in establishing the  result in the case $m=1$.
In Section~4 we use a different argument to prove two analogues of Theorem~\ref{thm:lowweight} in arbitrary weight; these results 
are the main input for the proofs of Theorems~\ref{thm:Atkin2} and \ref{thm:Atkin3}.    The last section contains the proofs of the three main theorems.

 \section*{Acknowledgments}
We thank Frank Calegari, Robert Dicks, J.-P. Serre, and the anonymous referee for helpful comments  on a previous version of this manuscript.

\section{Background}
\subsection{Modular forms}
If $f$ is a function on the upper half-plane,  $k\in \frac12\Z$,  and $\gamma=\ppmatrix abcd\in \GL_2^+(\Q)$, 
we define
\begin{gather*}
\(f\sk\gamma\)(\tau):=(\det\gamma)^\frac k2(c\tau+d)^{-k}f(\gamma\tau).
\end{gather*}
Given $k \in \frac{1}{2}\Z$, a positive integer $N$, a multiplier system $\nu$ in weight $k$ on $\Gamma_0(N)$, and a subring $A \subseteq \C$, 
we denote by $M_k(N, \nu, A)$, $S_k(N, \nu, A)$, and $M_k^!(N, \nu, A)$ the spaces of modular forms, cusp forms, and weakly holomorphic modular forms 
of weight $k$ and multiplier $\nu$ on $\Gamma_0(N)$ whose Fourier coefficients lie in $A$. For basic properties of mutiplier systems and modular forms one may consult for example \cite{knopp} and \cite{diamond_shurman}. These forms satisfy the transformation law
\begin{gather*}
f\sk \gamma=\nu (\gamma) f\ \ \ \text{for}\ \  \  \gamma=\ppmatrix abcd\in \Gamma_0(N)
\end{gather*}
as well as the appropriate cusp conditions (weakly holomorphic forms are allowed poles at the cusps).
 If $\nu=1$ we omit it from the notation.
If in addition $k$ is even,  we  let $S_k^{\operatorname{new}}(N, \C) \subseteq S_k(N, \C)$ denote the new subspace, and define
$S_k^{\operatorname{new}}(N,  A) := S_k(N,  A) \cap S_k^{\operatorname{new}}(N,  \C)$.

When $N$ is square-free, there is an Atkin-Lehner involution $W_p$ on $S_k(N, \C)$ for every prime divisor $p$ of $N$ \cite{atkin-lehner}. Given a 
tuple $\ep=(\ep_p)_{p\mid N}$ where each $\ep_p\in \{\pm1\}$, let $S_k^{\operatorname{new}}(N, \C, \ep)$ be the subspace consisting of those forms $f$ for which $f\big|_kW_p=\ep_p f$ for each prime $p\mid N$.

Throughout, $\ell\geq 5$ will denote a fixed prime number.
When $A$ is the subring of algebraic numbers that are integral at all primes above $\ell$, we omit it from the notation and simply write $M_k(N, \nu)$, $S_k(N, \nu)$, $M_k^!(N, \nu)$, and $S_k^{\operatorname{new}}(N)$.  
If in addition $\ell\nmid N$, then we write $S_k^{\operatorname{new}}(N, \varepsilon)$ for the subspace of $S_k^{\operatorname{new}}(N)$ attached to the tuple $\varepsilon$ (this makes sense since for such $\ell$, each 
  involution $W_p$ acts on $S_k^{\operatorname{new}}(N)$.)

The Dedekind eta function is defined by
\begin{gather*}
  \eta(\tau):=q^\frac1{24}\prod_{n=1}^\infty(1-q^n), \qquad q:=e^{2\pi i\tau}.
\end{gather*}
Then the  eta-multiplier $\nu_\eta$ is given by 
\begin{gather}\label{eq:etamult}
\eta(\gamma\tau)=\nu_\eta(\gamma)(c\tau+d)^\frac12\,\eta(\tau), \qquad \gamma=\pMatrix abcd\in \SL_2(\Z)\text{.}
\end{gather}

We collect some facts about modular forms which transform with a power of the eta-multiplier.  Proofs for some of the non-obvious facts can be found 
in Section 2 of \cite{ahlgren-beckwith-raum}.
If $f\in M_k\(1, \nu_\eta^r\)$,
then $\eta^{- r}f \in M^!_{k-\frac r2}(1)$.
It follows that $f$ has a Fourier expansion of the form
\begin{gather}
\label{eq:fourier-expansion}
  f
=
  \sum_{n\equiv  r\spmod{24}}
  a(n)q^\frac n{24}
\text{.}
\end{gather}
By (2.6) of  \cite{ahlgren-beckwith-raum} we have
\begin{gather}\label{eq:ktor}
M_k(1, \nu_\eta^r)=\{0\}\ \ \text{unless}\ \ r\equiv 2k\pmod 4.
\end{gather}
Let  $\nu_\theta$ be the multiplier on $\Gamma_0(4)$ attached to the theta function $\theta(\tau)=\sum q^{n^2}$
and define $f(\tau)\big|V_d:=f(d\tau)$.
By (2.7) of \cite{ahlgren-beckwith-raum} and  \eqref{eq:ktor} we have 
\begin{gather}\label{eq:etamultv24}
  f \in M_k\(1, \nu_\eta^r\)
\implies
  f\big| V_{24} \in M_k\(576,  \pmfrac{12}\bullet \nu_{\theta}^r \) =M_k\(576,  \pmfrac{12}\bullet \nu_{\theta}^{2k} \) 
\text{.}
\end{gather}
In particular $f\big| V_{24}$ is a modular form of half-integral weight in the sense of Shimura \cite{shimura}.

For each prime $Q\geq 5$ we have the Hecke operator
\begin{gather*}
  T(Q^2) \,:\, S_k\(1, \nu_\eta^r\)\rightarrow S_k\(1, \nu_\eta^r\)
\text{.}
\end{gather*}
If $f \in S_k\(1, \nu_\eta^r\)$  with $(r, 24)=1$ has Fourier expansion~\eqref{eq:fourier-expansion} then we have (see e.g. \cite[Proposition~11]{Yang_shimura})
\begin{gather}\label{eq:heckedef}
  f\big|T(Q^2)
=
  \sum \(
  a(Q^2n) + Q^{k-\frac32}\pmfrac{-1}{Q}^{k-\frac12}\pmfrac{12n}{Q} a(n) + Q^{2k-2}a\pmfrac{n}{Q^2}
  \)
  q^\frac n{24}.\end{gather}

For each squarefree $t$ with $(t, 6)=1$  there is a Shimura lift $\Sh_t$ on $S_k\(1, \nu_\eta^r\)$, defined via \eqref{eq:etamultv24} and the Shimura lift~\cite{shimura} on~$S_k\(576, \pfrac{12}\bullet \nu_\theta^{2k}\)$.
The action on Fourier expansions is given by 
\begin{gather}\label{eq:shimlift}
\Sh_t\(\sum a(n)q^\frac n{24}\)=\sum A_t(n)q^n,
\end{gather}
where 
\begin{gather}\label{eq:shimliftcoeff}
  A_t(n)
=
  \sum_{d\mid n}
  \pmfrac{-1}d^{k-\frac12}\pmfrac{12t}d d^{k-\frac32}\,
  a\pmfrac{tn^2}{d^2}
\text{.}
\end{gather}

Then we have (see \cite[(2.13)]{ahlgren-beckwith-raum})
\begin{gather}\label{eq:shimcong}
  f \equiv 0 \pmod\ell
\iff
  \Sh_t(f) \equiv 0 \pmod\ell
  \qquad \text{for all squarefree $t$.}
\end{gather}
From    work of Yang~\cite{Yang_shimura}
it follows  that we   have 
\begin{gather}\label{eq:shimliftimage}
  \Sh_t \,:\,
  S_k\(1, \nu_\eta^r\)
\longrightarrow
  S_{2k-1}^{\operatorname{new}}\(6, -\pmfrac{8}{r}, -\pmfrac{12}{r}\)\otimes \pmfrac{12}\bullet.
\end{gather}
Moreover, for all primes $Q\geq 5$ we have 
\begin{gather}\label{eq:shimliftcommute}
  \Sh_t\(f\big|T(Q^2)\)=\(\Sh_t f\)\big|T(Q)
\text{,}
\end{gather}
where $T(Q)$ is the Hecke operator of index $Q$ on the integral weight space.

The connection with partitions is given by the fundamental relationship 
\begin{gather*}
\frac1{\eta(\tau)}=\sum p \left( \frac{ n + 1}{24} \right) q^\frac{n}{24}.
\end{gather*}
For our applications there are 
 two important modular forms for each $\ell$ (see \cite[\S 2]{ahlgren-beckwith-raum} for the construction).
In particular, there is  a modular form  $f_{\ell}  \in S_{\frac{\ell - 2}{2}} \( 1, \nu_\eta^{-\ell}, \Z\)$ with
\begin{gather}\label{eq:fldef}
f_{\ell} \equiv \sum p \left( \frac{ \ell n + 1}{24} \right) q^\frac{n}{24} \pmod\ell.
\end{gather}
There is also a form  $g_{\ell}  \in S_{\frac{\ell^2 - 2}{2}} \(1, \nu_\eta^{-1}, \Z\)$ with 
\begin{gather}\label{eq:gldef}
g_{\ell} \equiv \sum_{\pfrac{-n}{\ell}=-1} p \left( \frac{  n + 1}{24} \right) q^\frac n{24} \pmod\ell.
\end{gather}
From the discussion in the Introduction, we have $g_\ell\not\equiv 0\pmod\ell$, and 
\begin{gather*}
f_\ell\equiv 0\pmod\ell \iff \ \ell=5, 7, 11.
\end{gather*}
From the discussion above,  each Shimura lift of $f_\ell$ is in the space 
\begin{gather*}
S_{\ell-3}^{\rm new}\(6, -\pmfrac{8}{-\ell}, -\pmfrac{12}{-\ell}\)\otimes \pmfrac{12}\bullet,
\end{gather*}
and  each Shimura lift of $g_\ell$ is in the space 
\begin{gather}\label{eq:shimg}
S_{\ell^2-3}^{\rm new}\(6, -1, -1\)\otimes \pmfrac{12}\bullet.
\end{gather}

\subsection{Modular Galois representations}

Let $k$ and $N$ be positive integers with $k$ even and $N$ coprime to $\ell$. 
We recall properties of the Galois representations attached to eigenforms in $S_k(N)$ that will be used in the next section. 
We recall that $\ell \ge 5$.

We let $\Qbar \subseteq \C$ be the algebraic closure of $\Q$ in $\C$. 
For each prime $p$, we fix an algebraic closure $\Qbar_p$ of $\Q_p$ and an embedding $\iota_p \colon \Qbar \hookrightarrow \Qbar_p$. 
Via $\iota_p$, we view $G_p := \Gal(\Qbar_p/\Q_p)$ as a subgroup of $G_\Q := \Gal(\Qbar/\Q)$. 
We let $I_p \subseteq G_p$ denote the inertia subgroup and let $\Frob_p \in G_p/I_p$ denote the arithmetic Frobenius. 
We view the coefficients of any $f \in S_k(N)$ as elements of $\Qbar_\ell$ via $\iota_\ell$. 

We denote by $\chi \colon G_{\Q} \to \Z_{\ell}^{\times}$ (resp. $\omega \colon G_{\Q} \to \F_\ell^\times$) the $\ell$-adic (resp. the mod $\ell$) cyclotomic character, and similarly with $G_{\Q}$ replaced by $G_K$ for $K/\Q$ finite or by $G_p$ with $p$ a prime, etc. 
We let $\omega_2, \omega'_2 \colon I_\ell \to \F_{\ell^2}^\times$ denote Serre's fundamental characters of level $2$ \cite[\S2.1]{serre-duke}.
These are characters of order $\ell^2 - 1$ with  $\omega_2^\ell = \omega_2'$, $\omega_2'^\ell = \omega_2$, and $\omega_2^{\ell+1} = \omega_2'^{\ell+1} = \omega$.

 The next theorem summarizes some important properties of modular Galois representations, and is due to many people, including Deligne, Fontaine, Langlands, Ribet, and Shimura. 
 See \cite[\S3.2.2]{hida-mod-forms-gal-cohom} and \cite[\S2]{edixhoven-weight} for more details and references. 
\begin{theorem}\label{gal-rep}
	Let $f = q + \sum_{n\ge 2} a_n q^n \in S_k(N)$ be a normalized Hecke eigenform.
	There is a continuous irreducible representation $\rho_f \colon G_{\Q} \to \GL_2(\Qbar_\ell)$ with semisimple mod $\ell$ reduction
	$\rhobar_f \colon G_{\Q} \to \GL_2(\Fbar_\ell)$
	satisfying the following properties. 
	\begin{enumerate}
		\item\label{gal-rep:unramified} If $p\nmid \ell N$, then $\rho_f$ is unramified at $p$ and the characteristic polynomial of $\rho_f(\Frob_p)$ is $X^2 - \iota_\ell(a_p) X + p^{k-1}$. This uniquely characterizes $\rho_f$. In particular, $\det\rho_f = \chi^{k-1}$.
		
		\item\label{gal-rep:at-q} If $q\mid N$ and $q^2 \nmid N$, then $\rho_f|_{I_q}$ is unipotent. In particular, the prime-to-$\ell$ Artin conductor $N(\rhobar_f)$ of $\rhobar_f$ is not divisible by $q^2$.
		
		\noindent If further $f$ is $q$-new, then $\rho_f|_{G_q}$ is an extension of $\psi$ by $\chi\psi$ where $\psi \colon G_q \to \Qbar_\ell^\times$ is the unramified character with $\psi(\Frob_q)=\iota_\ell(a_q)$.

		\item\label{gal-rep:at-l} Assume that $2 \le k \le \ell+1$.   Then
		\begin{itemize}
			\item If $\iota_\ell(a_\ell)$ is an $\ell$-adic unit, 
			then $\rho_f|_{G_\ell}$ is reducible and $\rho_f|_{I_\ell}$ is an extension of the trivial character by $\chi^{k-1}$.
			\item If $\iota_\ell(a_\ell)$ is not an $\ell$-adic unit, then $\rhobar_f|_{G_\ell}$ is irreducible and $\rhobar_f|_{I_\ell} \cong \omega_2^{k-1} \oplus \omega_2'^{(k-1)}$.
		\end{itemize}
	\end{enumerate}
\end{theorem}

Although we have suppressed it from the notation, we note that the representations $\rho_f$ and $\rhobar_f$ do depend on the choice of embedding $\iota_\ell \colon \Qbar \hookrightarrow \Qbar_\ell$, at least up to the prime that it determines in the coefficient field $\Q(\{a_n\}_{n\ge 1}) \subseteq \Qbar$ of $f$.  

\subsection{Congruences for non-eigenforms}\label{sec:cong-non-eig}
We will ultimately be interested in congruences for cusp forms which are not necessarily eigenforms. 
To do so, we record two lemmas  which allow us to bootstrap from the case of newforms to the general case.

Let $k$ and $N$ be positive integers with $k$ even and $N$ coprime to $\ell$. 
Let $f_1, \ldots, f_n \in S_k^{\operatorname{new}}(N)$ and write
\begin{gather}\label{eq:lincombnewform}
	f_j=\sum_{i=1}^d c_{i, j} g_i,
\end{gather}
with newforms $g_1,\ldots,g_d \in S_k^{\operatorname{new}}(N)$  and $c_{i,j} \in \Qbar$. 
Let $E$ be a number field which contains
the coefficients of each $g_i$ as well as all of the coefficients $c_{i, j}$.
Fix a prime  $\lambda$  of $E$ over $\ell$ and let $e$ be its ramification index. 
Define
\begin{gather}\label{eq:mfs}
	m(f_1,\ldots,f_n) := \max(0,  - \min(\ord_{\lambda}(c_{i, j}))).
\end{gather}
With this notation we have the following.
\begin{lemma} \label{lem:congruence} 
	Let the notation be as above and let $m \ge 1$ be an integer.
	Assume that there is a prime $Q \nmid N\ell$, an integer $a$, and an integer $r \ge 1$ such that 
	$g_i\big |T(Q)\equiv a g_i\bmod\lambda^r$ for each $1\le i \le d$. 
	If $r \ge m + m(f_1,\ldots,f_n)$, then for  $1\le j \le n$ we have
	\begin{gather*}
		f_j\big | T(Q)\equiv a f_j \mod \lambda^m.
	\end{gather*}
	If further $f_1,\ldots,f_n \in S_k^{\operatorname{new}}(N,\mathbb{Z})$ and $r \ge em + m(f_1,\ldots,f_n)$ then for  $1\le j \le n$ we have
	\begin{gather*}
		f_j\big | T(Q)\equiv a f_j \pmod {\ell^m}.
	\end{gather*}
\end{lemma}

\begin{proof}
	This follows immediately from the definitions of $m(f_1,\ldots,f_n)$ and ramification index.
\end{proof}

\begin{lemma}\label{lem:from-lambda-to-ell}
	Let $m \ge 1$ be an integer. 
	Let $Q \nmid N\ell$ be prime,  let $a$  be an integer, and let $r \ge 1$ be an integer such that $g_i\big |T(Q)\equiv a g_i\bmod{\lambda^r}$ for every newform $g \in S_k^{\operatorname{new}}(N)$. 
	There is an integer $c \ge 0$, depending on $\ell$, $k$, and $N$ but not on $Q$, $a$, or $m$, such that if $r \ge em + c$, then 
	\begin{gather*}
		f\big |T(Q)\equiv a f\pmod {\ell^m}
	\end{gather*}
	for all $f \in S_k^{\operatorname{new}}(N)$.
\end{lemma}

\begin{proof}
	Let $A \subset \Qbar$ be the subring of elements that are integral at all primes above $\ell$ (so $S_k^{\operatorname{new}}(N) = S_k^{\operatorname{new}}(N, A)$).
	Let $h_1,\ldots,h_d $ generate $S_k^{\operatorname{new}}(N)$ over $A$ and let $f_1,\ldots,f_n$ be a basis for $S_k^{\operatorname{new}}(N, \Z)$ (in fact we can take $n = d$).
	Since $S_k^{\operatorname{new}}(N, \Z)$ generates $S_k^{\operatorname{new}}(N, \Qbar)$ over $\Qbar$, we can write 
	\begin{gather}
		h_j=\sum_{i=1}^n d_{i,j} f_i,
	\end{gather}
	with $d_{i,j} \in \Qbar$. 
	Choose $c_0 \ge 0$ such that $\ell^{c_0} d_{i,j} \in A$ for each $i,j$. 
	To show that $f\big |T(Q)\equiv a f\pmod {\ell^m}$ for all $f \in S_k^{\operatorname{new}}(N)$, it suffices to show that $f_i\big | T(Q)\equiv a f_i \pmod {\ell^{m + c_0}}$ for each $1 \le i \le n$. 
	Setting $c = ec_0+  m(f_1,\ldots,f_n) $ with $m(f_1,\ldots,f_n)$ as in \eqref{eq:mfs}, the lemma now follows from Lemma~\ref{lem:congruence}.
\end{proof}

\begin{remark}\label{rmk:lambda-to-ell-epsilon}
	When $N$ is squarefree, the same result as Lemma~\ref{lem:from-lambda-to-ell} holds, replacing $S_k^{\operatorname{new}}(N)$ with $S_k^{\operatorname{new}}(N, \varepsilon)$. 
\end{remark}

\section{Congruences in low weight}
In this section, we use Galois representations together with some group theoretic arguments to prove Theorem \ref{thm:lowweight}, which implies Theorem \ref{thm:Atkin1}. 
The proof is by induction with the bulk of the work devoted to proving the base case, which is essentially the combination of Propositions~\ref{big-image} and~\ref{congruences}. 
A key technical lemma is Lemma~\ref{galois-extension}, which roughly states that the mod $\ell$ Galois representations associated to sufficiently distinct normalized eigenforms cut out sufficiently disjoint field extensions. 
This lemma and its proof are inspired by and similar to that of \cite[Lemma~7.1.5(3)]{10author}.

We recall the standing assumption that $\ell\geq 5$ is prime.
We again let $k$ and $N$ be positive integers with $k$ even and $N$ coprime to $\ell$, and further assume that $N$ is squarefree.
Recall that we have a fixed embedding $\iota_\ell \colon \Qbar \hookrightarrow \Qbar_\ell$ that allows us to view Fourier coefficients of modular forms in $S_k(N)$ as elements of $\Qbar_\ell$.  

The following two lemmas are consequences of Theorem \ref{gal-rep} and give us information on the image of the mod $\ell$ Galois representation associated to a normalized eigenform in $S_k(N)$.

\begin{lemma}\label{twist}
	Let $f,g \in S_{k}(N)$ be normalized eigenforms such that $\rhobar_g \cong \rhobar_f \otimes \eta$ for some nontrivial continuous character $\eta \colon G_{\Q} \to \Fbar_\ell^\times$.
	Then 
	\begin{enumerate}
		\item\label{twist:cyclotomic} $\eta = \omega^{\frac{\ell -1}{2}}$.
		\item\label{twist:nontrivial} Assume further that $k \le \ell-1$.
		Letting $a_\ell$ denote the $\ell$-th Fourier coefficient of $f$, we have
		\begin{itemize}
			\item If $\iota_\ell(a_\ell)$ is an $\ell$-adic unit, then $k = \frac{\ell +1}{2}$.
			\item If $\iota_\ell(a_\ell)$ is not an $\ell$-adic unit, then $k = \frac{\ell +3}{2}$.
		\end{itemize}
	\end{enumerate}
\end{lemma}

\begin{proof}
	To prove that $\eta = \omega^i$ for some $1 \le i \le \ell - 2$, it suffices to show that $\eta$ is unramified outside of $\ell$. 
	Take any prime $p \ne \ell$. 
	Since we are assuming that $N$ is squarefree, parts~\ref{gal-rep:unramified} and~\ref{gal-rep:at-q} of Theorem~\ref{gal-rep} imply that 
	\begin{gather*} \rhobar_f|_{I_p} \cong \begin{pmatrix} 1 & \ast \\ & 1 \end{pmatrix},\end{gather*}
	and similarly for $\rhobar_g|_{I_p}$. 
	Since $\rhobar_g \cong \rhobar_f \otimes \eta$, we must have $\eta|_{I_p} = 1$. 
	To see that $i = \frac{\ell-1}{2}$, we use that
		\begin{gather*} \det\rhobar_f = \omega^{k-1} = \det\rhobar_g = \eta^2\det\rhobar_f, \end{gather*}
	so $\eta = \omega^i$ is quadratic.
	
	For part~\ref{twist:nontrivial}, we use part~\ref{gal-rep:at-l} of Theorem~\ref{gal-rep}. 
	If $\iota_\ell(a_\ell)$ is an $\ell$-adic unit, it implies that $\omega^{k-1} = \omega^{\frac{\ell-1}{2}}$, so $k = \frac{\ell+1}{2}$. 
	If $\iota_\ell(a_\ell)$ is not an $\ell$-adic unit, it implies that 
		\begin{gather*} \omega_2^{k-1} = (\omega_2')^{k-1}\omega^{\frac{\ell-1}{2}} = \omega_2^{\ell(k-1) + \frac{\ell^2-1}{2}}. \end{gather*}
	Since $\omega_2$ has order $\ell^2-1$, we obtain $k = \frac{\ell+3}{2}$.
\end{proof}

\begin{lemma}\label{irred}
	Let $f = q + \sum_{n\ge 2} a_n q^n \in S_k(N)$ be a normalized eigenform. 
	Assume that $2 \le k \le \ell-1$ and that there is a prime $q \mid N$ such that $f$ is $q$-new and $q^{k-1} \not\equiv q^{\pm 1} \pmod \ell$. 
	Then the following are true.
	\begin{enumerate}
		\item\label{irred:not-eis} $\rhobar_f$ is irreducible.
		\item\label{irred:not-cm} Assume there is a quadratic extension $K/\Q$ such that $\rhobar_f|_{G_K}$ is reducible. 
		Then $\rhobar_f \cong \rhobar_f \otimes \omega^{\frac{\ell-1}{2}}$ and
		\begin{itemize}
			\item If $\iota_\ell(a_\ell)$ is an $\ell$-adic unit, then $k = \frac{\ell+1}{2}$.
			\item If $\iota_\ell(a_\ell)$ is not an $\ell$-adic unit, then $k = \frac{\ell+3}{2}$.
		\end{itemize}
	\end{enumerate}
\end{lemma}

\begin{proof}
To establish  part \ref{irred:not-eis}, 
	assume that $\rhobar_f \cong \psi_1 \oplus \psi_2$ for characters $\psi_i \colon G_{\Q} \to \Fbar_\ell^\times$. 
	Since $N$ is squarefree, $N(\rhobar_f) \mid N$ is also squarefree and for every prime $p \ne \ell$, at most one of $\psi_1,\psi_2$ is ramified at $p$ (see \cite[\S1.1]{carayol-modl}, for example). 
	On the other hand, $\det\rhobar_f = \omega^{k-1}$, so $\psi_1\psi_2$ is unramified at all primes $p \ne \ell$. 
	It follows that $\psi_1 = \omega^a$ and $\psi_2 = \omega^b$ for some $0 \le a,b \le \ell-2$. 
	Reordering if necessary, part~\ref{gal-rep:at-l} of Theorem~\ref{gal-rep} implies that $\psi_1 = \omega^{k-1}$ and $\psi_2 = 1$.
	This together with our assumption on $q$ contradicts part~\ref{gal-rep:at-q} of Theorem~\ref{gal-rep}.
	
	We turn to part~\ref{irred:not-cm}. 
	Since $\rhobar_f$ is irreducible we must have  $\rhobar_f|_{G_K} \cong \psi_1 \oplus \psi_2$ for nontrivial characters $\psi_1,\psi_2 \colon G_K \to \Fbar_\ell^\times$ which are permuted by $\Gal(K/\Q)$. 
	It follows that $\rhobar_f$ is the induction of a character $\psi \colon G_K \to \Fbar_\ell^\times$ and letting $\eta$ be the quadratic character of $K/\Q$, that $\rhobar_f \cong \rhobar_f \otimes \eta$. 
	We then apply part~\ref{twist:nontrivial} of Lemma~\ref{twist}.
\end{proof}

We apply the preceding two lemmas to our particular space of interest $S_{\ell-3}^{\mathrm{new}}(6)$.

\begin{proposition}\label{big-image}
Recall that $\ell \ge 5$. For any newform $f \in S_{\ell-3}^{\mathrm{new}}(6)$, the image of $\rhobar_f$ contains a conjugate of $\SL_2(\F_\ell)$.
\end{proposition}

\begin{proof}
We first note that $S_{2}(6) = \{0\}$, so we may assume that $\ell \ge 7$. 

By \cite[Theorem~2.47(b)]{darmon-diamond-taylor}, there are four possibilities for the image of $\rhobar_f$: 
\begin{enumerate}
	\item $\rhobar_f$ is reducible.
	\item $\rhobar_f$ is dihedral, i.e. $\rhobar_f$ is irreducible but $\rhobar_f|_{G_K}$ is reducible for some quadratic $K/\Q$.
	\item $\rhobar_f$ is exceptional, i.e. the projective image of $\rhobar_f$ is conjugate to one of $A_4$, $S_4$, or $A_5$.
	\item The image of $\rhobar_f$ contains a conjugate of $\SL_2(\F_\ell)$.
\end{enumerate}
We rule out each of the first three possibilities in turn.

First, $2^{\ell - 4} \not\equiv 2^{\pm 1} \pmod \ell$ for any $\ell \ge 7$, so part~\ref{irred:not-eis} of Lemma~\ref{irred} shows that $\rhobar_f$ is irreducible. 
If further $\ell \ge 11$, then $\ell - 3 \ne \frac{\ell+1}{2}, \frac{\ell+3}{2}$, so part~\ref{irred:not-cm} of Lemma~\ref{irred} shows that $\rhobar_f$ is not dihedral. 
For $\ell=7$, the space $S_4^{\mathrm{new}}(6)$ is one-dimensional and  spanned by the newform 
\begin{gather}\label{lequals7} 
	f=q-2q^2-3q^3+4q^4+6q^5+6q^6-16q^7+\cdots 
\end{gather}
by \cite{lmfdb}.
If $\rhobar_f$ were dihedral, then part~\ref{irred:not-cm} of Lemma~\ref{irred} would imply that $\rhobar_f \cong \rhobar_f \otimes\omega^3$. 
Since $\omega^3(\Frob_5) = -1$, this would imply that $\operatorname{tr}\rhobar_f(\Frob_5) = 0$. 
But $a_5 = 6 \not\equiv 0 \pmod 7$, a contradiction.

It remains to rule out the exceptional case, and for this it suffices to show that the projective image contains an element of order $\ge 6$. 
Since the characters $\omega$ and $(\omega_2/\omega_2')$ have orders $\ell-1$ and $\ell+1$, respectively, 
the description of $\rhobar_f|_{I_\ell}$ in part~\ref{gal-rep:at-l} of Theorem~\ref{gal-rep} implies that the projective image of $\rhobar_f$ contains an element of order $\frac{\ell-1}{\gcd\{\ell-1, \ell-4\}}$ if $\iota_\ell(a_\ell)$ is a unit, and an element of order $\frac{\ell+1}{\gcd\{\ell+1, \ell-4\}}$ if $\iota_\ell(a_\ell)$ is not a unit. 
These are both $\ge 6$ if $\ell = 11$, $17$, or $\ell \ge 23$.
When $\ell = 19$, there are three newforms in $S_{16}^{\mathrm{new}}(6)$ with LMFDB labels 6.16.a.a, 6.16.a.b, and 6.16.a.c. 
The values of $a_\ell$ for these three newforms are $2163188180$, $4934015444$, and $-5895116260$, respectively. 
In each case $a_\ell \not\equiv 0 \pmod {19}$ and $\frac{18}{\gcd\{18,15\}} = 6$, so part~\ref{gal-rep:at-l} of Theorem~\ref{gal-rep} again shows that $\rhobar_f$ cannot be exceptional when $\ell = 19$.

For $\ell = 7,13$, we can rule out the possibility of exceptional image as follows. 
Say we have a newform $f = q + \sum_{n \ge 2} a_n q^n$ in $S_{\ell -3}^{\mathrm{new}}(6)$ and a prime $p \nmid 6\ell$ such that $a_p \in \Z$. 
Setting $u(p):=a_p^2/p^9$, 
if the projective image of $\rhobar_f$ is $A_4, S_4$, or $A_5$, then we have 
\begin{gather}\label{exceptional}
	u(p) \equiv 4,0,1,2 \pmod \ell \quad \text{or} \quad u(p)^2-3u(p)+1 \equiv 0 \pmod \ell,
\end{gather}	
depending on whether the image of $\rhobar_f(\Frob_p)$ in $\PGL_2(\overline{\F}_\ell)$ has order 1,2,3,4, or 5 (see for example \cite[p. 189]{Ribet}). 
When $\ell = 7$, one can check directly that \eqref{exceptional} is not satisfied when $p = 5$ for the unique newform $f \in S_{4}^{\mathrm{new}}(6)$ with Fourier expansion \eqref{lequals7} above. 
When $\ell = 13$, there is again a unique newform $f \in S_{10}^{\mathrm{new}}(6)$ and it has Fourier expansion 
\begin{gather*}f=q-16 q^2+81 q^3+256 q^4+2694 q^5+\cdots\end{gather*}
by \cite{lmfdb}. 
We can again check directly that \eqref{exceptional} is not satisfied when $p = 5$. 
\end{proof}

\begin{remark}\label{atkin1vs2rmk}
	The conclusion of Proposition~\ref{big-image} does not hold in general for the spaces $S_{\ell^2 - 3}^{\mathrm{new}}(6)$;
	this is the main reason that we are able to say more about Type I congruences than those of Type II.
\end{remark}

Next we will prove a few group theoretic lemmas (Lemmas \ref{disjoint-im}, \ref{simple-groups} and \ref{galois-extension}) leading to Proposition \ref{congruences}.
Before continuing we need to introduce some notation.
\begin{notation}\label{proj-and-twisted-reps}
	Let $G$ be a group and let $\tau \in \Gal(\Fbar_\ell/\F_\ell)$.
	\begin{itemize}
		\item For a homomorphism $\rhobar \colon G \to \GL_2(\overline{\F}_\ell)$, we write ${}^\tau\rhobar \colon G \to \GL_2(\Fbar_\ell)$ for the composite of $\rhobar$ with the automorphism $\GL_2(\Fbar_\ell) \to \GL_2(\Fbar_\ell)$ induced by $\tau \colon \Fbar_\ell \to \Fbar_\ell$.
		\item Similarly, for a homomorphism $r \colon G \to \PGL_2(\overline{\F}_\ell)$, we write ${}^\tau r \colon G \to \PGL_2(\Fbar_\ell)$ for the composite of $r$ with the automorphism $\PGL_2(\Fbar_\ell) \to \PGL_2(\Fbar_\ell)$ induced by $\tau \colon \Fbar_\ell \to \Fbar_\ell$.
		\item For two homomorphisms $r_1,r_2 \colon G \to \PGL_2(\Fbar_\ell)$, we write $r_1 \cong r_2$ if they are conjugate by an element of $\PGL_2(\Fbar_\ell)$.
		\item For a normalized eigenform $f \in S_k(N)$, we let $r_f \colon G_\Q \to \PGL_2(\Fbar_\ell)$ be the composite of $\rhobar_f \colon G_\Q \to \GL_2(\Fbar_\ell)$ with the projection $\GL_2(\Fbar_\ell) \to \PGL_2(\Fbar_\ell)$. 
	\end{itemize}
\end{notation}

\begin{lemma}\label{proj-im}
	Let $f,g \in S_{k}(N)$ be normalized eigenforms.
	If $r_f \cong r_g$ and $\rhobar_f \not\cong \rhobar_g$, then $\rhobar_f \cong \rhobar_g \otimes\omega^{\frac{\ell-1}{2}}$.
\end{lemma}

\begin{proof}
	Conjugating if necessary, we can assume that  $r_f = r_g$. 
	Then we can define a character $\eta \colon G_{\Q} \to \Fbar_\ell^\times$ by $\eta(\sigma) = \rhobar_f(\sigma)\rhobar_g(\sigma)^{-1}$, and we have $\rhobar_f \cong \rhobar_g \otimes \eta$.  
	The lemma now follows from Lemma~\ref{twist}.
\end{proof}

\begin{lemma}\label{disjoint-im}
	For $i = 1, 2$, let $\mathbb{F}_{q_i}/\F_\ell$ be the field of cardinality $q_i$ in $\Fbar_\ell$, with $q_i$ some power of $\ell$, and let $r_i \colon G_{\Q} \to \PGL_2(\Fbar_\ell)$ be a continuous homomorphism with image containing $\PSL_2(\mathbb{F}_{q_i})$ and contained in $\PGL_2(\mathbb{F}_{q_i})$.
	Let $L_i$ be the subfield of $\Qbar$ fixed by $\ker(r_i)$ and let $K_i/\Q$ be the subextension of $L_i/\Q$ such that $\Gal(L_i/K_i) \cong \PSL_2(\mathbb{F}_{q_i})$.
	Then the following are equivalent:
	\begin{enumerate}
		\item\label{disjoint-im:not-ab} $L_1 \cap L_2 \not\subseteq K_1K_2$.
		\item\label{disjoint-im:equal} $L_1 = L_2$.
		\item\label{disjoint-im:conj} There is $\tau \in \Gal(\Fbar_\ell/\F_\ell)$ such that $r_1 \cong {}^\tau r_2$. 
	\end{enumerate}
\end{lemma}

\begin{proof}
	First note that for each $i = 1,2$, $\PSL_2(\mathbb{F}_{q_i})$ is simple since $\lvert \mathbb{F}_{q_i} \rvert \ge 4$ (recall that $\ell \ge 5$).
	
	Clearly, \eqref{disjoint-im:equal} implies \eqref{disjoint-im:not-ab}. 
	On the other hand, since $\PSL_2(\mathbb{F}_{q_i})$ is simple and is the unique nontrivial proper normal subgroup of $\PGL_2(\mathbb{F}_{q_i})$, if $L_1 \cap L_2 \not\subseteq K_1K_2$, then $L_1 \subseteq L_2$ or $L_2 \subseteq L_1$, and $K_1 = K_2$.
	In either case, by comparing Jordan--Holder factors, we must have $L_1 = L_2$. 
	
	It is immediate that \eqref{disjoint-im:conj} implies \eqref{disjoint-im:equal}.
	Conversely, if $L_1 = L_2$, then there is an isomorphism of groups $\phi \colon r_1(G_{\Q}) \cong r_2(G_{\Q})$. 
	In particular, this implies that $q_1 = q_2$.
	So letting $\F_q = \F_{q_1} = \F_{q_2}$, we have either that $r_1(G_{\Q}) = r_2(G_{\Q}) = \PSL_2(\F_q)$ and $\phi$ is an automorphism of $\PSL_2(\F_q)$, 
	or that $r_1(G_{\Q}) = r_2(G_{\Q}) = \PGL_2(\F_q)$ and $\phi$ is an automorphism of $\PGL_2(\F_q)$. 
	The automorphism group of both $\PSL_2(\F_q)$ and $\PGL_2(\F_q)$ is $\PGL_2(\F_q) \rtimes \Gal(\F_q/\F_\ell)$ (see \cite[Theorem 30]{steinberg1967lectures}), which implies that $r_1$ is conjugate to ${}^\tau r_2$ for some $\tau \in \Gal(\F_q/\F_\ell)$. 
\end{proof}

\begin{lemma}\label{simple-groups}
	Let $G_1,\ldots,G_s$ and $H$ be simple nonabelian groups. 
	Any surjective homomorphism $f \colon \prod_{i=1}^s G_i \to H$ factors through some projection $p_j \colon \prod_{i=1}^s G_i \to G_j$.
\end{lemma}

\begin{proof}
	For each $1\le n \le s$, let $\iota_n \colon G_n \to \prod_{i=1}^s G_i$ be the canonical injection. 
	Assume that $f$ does not factor through any $p_j$. 
	Then there are indices $1\le m \ne n \le s$ such that $f\circ \iota_m \colon G_m \to H$ and $f \circ\iota_n \colon G_n \to H$ are nontrivial. 
	Since $G_n$ and $G_m$ are simple, $f\circ\iota_m$ and $f\circ\iota_n$ are injective. 
	Since $\iota_m(G_m)$ and $\iota_n(G_n)$ are normal subgroups of $\prod_{i=1}^s G_i$ and $f$ is surjective, $f\circ\iota_m(G_m)$ and $f\circ\iota_n(G_n)$ are normal subgroups of $H$. 
	We see that $f\circ\iota_m$ and $f\circ\iota_n$ are isomorphisms. 
	Since $H$ is nonabelian, we can then choose $x \in G_m$ and $y \in G_n$ such that $f\circ\iota_m(x)$ and $f\circ\iota_n(y)$ do not commute. 
	But $\iota_m(x)$ and $\iota_n(y)$ commute in $\prod_{i=1}^s G_i$, which gives a contradiction.
\end{proof}

\begin{lemma}\label{galois-extension}
	Let $f_1,\ldots,f_s \in S_k(N)$ be normalized eigenforms. 
	For each $1\le i \le s$,   assume that $\rhobar_{f_i}(G_{\Q})$ contains a conjugate of $\SL_2(\F_\ell)$, 
	and  let $M_i$ be the subfield of $\Qbar$ fixed by $\ker(\rhobar_{f_i})$.   Then we have the following.
	\begin{enumerate}
		\item\label{galois-extension:one-form} For each $1\le i \le s$, there is an extension $K_i/\Q$ of degree at most $2$ contained in $M_i$ and a finite extension $\mathbb{F}_{q_i}/\F_\ell$ such that $\rhobar_{f_i}(G_{K_i(\zeta_\ell)})$ is conjugate to $\SL_2(\mathbb{F}_{q_i})$.
		\item\label{galois-extension:multiple-forms} Assume moreover that for each $1\le i \ne j \le s$, there is no $\tau \in \Gal(\Fbar_\ell/\F_\ell)$ such that $r_{f_i} \cong {}^\tau r_{f_j}$. 
		Let $K_i$ and $\mathbb{F}_{q_i}$ be as in part~\ref{galois-extension:one-form}, and set $M = M_1\cdots M_s$ and $K = K_1\cdots K_s$. 
		Then $\Gal(M(\zeta_\ell)/K(\zeta_\ell)) \cong \prod_{i=1}^s \SL_2(\mathbb{F}_{q_i})$.
	\end{enumerate}
\end{lemma}

\begin{proof}
	For each $1 \le i \le s$, since $\rhobar_{f_i}(G_{\Q})$ contains a conjugate of $\SL_2(\F_\ell)$, \cite[Theorem~2.47(b)]{darmon-diamond-taylor} implies that the image of $r_{f_i}$ is conjugate to either $\PSL_2(\mathbb{F}_{q_i})$ or $\PGL_2(\mathbb{F}_{q_i})$ for some finite extension $\mathbb{F}_{q_i}/\F_\ell$. 
	Replacing $\rhobar_{f_i}$ by a conjugate if necessary, we assume that $r_{f_i}$ has image either $\PSL_2(\mathbb{F}_{q_i})$ or $\PGL_2(\mathbb{F}_{q_i})$. 
	We then let $K_i/\Q$ be the extension of degree at most $2$ such that $r_{f_i}(G_{K_i}) = \PSL_2(\mathbb{F}_{q_i})$, which is a simple group since $\ell \ge 5$. 
	Then $r_{f_i}(G_{K_i}(\zeta_\ell)) = \PSL_2(\mathbb{F}_{q_i})$ as well, so $\rhobar_{f_i}(G_{K_i(\zeta_\ell)})$ is a subgroup of $\overline{\F}_\ell^\times \SL_2(\mathbb{F}_{q_i})$ which contains $\SL_2(\mathbb{F}_{q_i})$ and  has trivial determinant. 
	It follows that $\rhobar_{f_i}(G_{K_i(\zeta_\ell)}) = \SL_2(\mathbb{F}_{q_i})$.
	
	To prove part~\ref{galois-extension:multiple-forms}, we first establish some preliminaries.
	For each $1\le i \le s$, let $L_i$ be the subfield of $M_i$ fixed by $\ker(r_{f_i})$. 
	We claim the following.
	\begin{enumerate}[label=(\alph*)]
		\item\label{Liext} $\Gal(L_iK(\zeta_\ell)/K(\zeta_\ell)) \cong \PSL_2(\mathbb{F}_{q_i})$.
		\item\label{Miext} $\Gal(M_iK(\zeta_\ell)/K(\zeta_\ell)) \cong \SL_2(\mathbb{F}_{q_i})$.
		\item\label{Lidisjoint} For $1 \le i \ne j \le s$, the fields $L_iK(\zeta_\ell)$ and $L_jK(\zeta_\ell)$ are disjoint over $K(\zeta_\ell)$.
	\end{enumerate}
	Since $\Gal(L_i/K_i) \cong \PSL_2(\mathbb{F}_{q_i})$ is nonabelian and simple and $K(\zeta_\ell)/K_i$ is abelian, these extensions are disjoint and $\Gal(L_iK(\zeta_\ell)/K(\zeta_\ell)) \cong \PSL_2(\mathbb{F}_{q_i})$, which gives claim~\ref{Liext}. 
	Further, since the unique nontrivial proper quotient of $\Gal(M_iK_i(\zeta_\ell)/K_i(\zeta_\ell)) \cong \SL_2(\mathbb{F}_{q_i})$ is $\Gal(L_iK_i(\zeta_\ell)/K_i(\zeta_\ell)) \cong \PSL_2(\mathbb{F}_{q_i})$, we have 
	\begin{gather*} M_iK_i(\zeta_\ell) \cap K(\zeta_\ell) \ne K_i(\zeta_\ell) \Leftrightarrow L_iK_i(\zeta_\ell) \cap K(\zeta_\ell) \ne K_i(\zeta_\ell). \end{gather*}
	So $M_iK_i(\zeta_\ell)$ and $K(\zeta_\ell)$ are also disjoint over $K_i(\zeta_\ell)$, which proves claim~\ref{Miext}. 
	We now prove claim~\ref{Lidisjoint}. 
	If this were not the case, we would have $L_i \subseteq L_jK(\zeta_\ell)$ or $L_j \subseteq L_iK(\zeta_\ell)$. 
	Without loss of generality, assume that $L_i \subseteq L_jK(\zeta_\ell)$. 
	But $\Gal(L_jK(\zeta_\ell)/\Q) \hookrightarrow \Gal(L_j/\Q) \times \Gal(K(\zeta_\ell)/\Q)$ has a unique nonabelian Jordan--Holder factor, namely the one isomorphic to $\Gal(L_j/K_j)$. 
	So $L_i \subseteq L_jK(\zeta_\ell)$ implies that $\Gal(L_iL_j/\Q)$ also has a unique nonabelian Jordan--Holder factor, which implies that $L_i \cap L_j \not\subseteq K_iK_j$. 
	By Lemma~\ref{disjoint-im}, this contradicts our hypotheses on $r_{f_i}$ and $r_{f_j}$. 
	
	To conclude, we prove by induction on $1 \le j \le s$ that 
	\begin{gather*} \Gal(M_1\cdots M_jK(\zeta_\ell)/K(\zeta_\ell)) \cong \prod_{i=1}^j \Gal(M_iK(\zeta_\ell)/K(\zeta_\ell)) \cong \prod_{i=1}^j \SL_2(\mathbb{F}_{q_i}).\end{gather*}
	The $j = 1$ case follows from claim~\ref{Miext} of the previous paragraph. 
	Now take $1 \le j-1 \le s$, and assume that 
	\begin{gather*} \Gal(M_1\cdots M_{j-1}K(\zeta_\ell)/K(\zeta_\ell)) \cong \prod_{i=1}^{j-1} \Gal(M_iK(\zeta_\ell)/K(\zeta_\ell)) \cong \prod_{i=1}^{j-1} \SL_2(\mathbb{F}_{q_i}).\end{gather*}
	We want to show that $M_1\cdots M_{j-1}K(\zeta_\ell)$ and $M_jK(\zeta_\ell)$ are disjoint over $K(\zeta_\ell)$. 
	Assume otherwise. 
	Since $\Gal(L_jK(\zeta_\ell)) \cong \PSL_2(\F_{q_j})$ is the unique nontrivial proper quotient of $\Gal(M_jK(\zeta_\ell)/K(\zeta_\ell)) \cong \SL_2(\F_{q_j})$, we must then have $L_jK(\zeta_\ell) \subseteq M_1\cdots M_{j-1}K(\zeta_\ell)$. 
	Since $\Gal(L_jK(\zeta_\ell)/K(\zeta_\ell))$ is nonabelian simple, the centre of $\prod_{i=1}^{j-1} \SL_2(\mathbb{F}_{q_i})$ maps trivially under the surjective morphism
	\begin{gather*} \prod_{i=1}^{j-1} \SL_2(\mathbb{F}_{q_i}) \cong \Gal(M_1\cdots M_{j-1}K(\zeta_\ell)/K(\zeta_\ell)) \to \Gal(L_jK(\zeta_\ell)/K(\zeta_\ell)), \end{gather*}
	and this map factors through
	\begin{gather*} \prod_{i=1}^{j-1} \PSL_2(\mathbb{F}_{q_i}) \cong \prod_{i=1}^{j-1} \Gal(L_iK(\zeta_\ell)/K(\zeta_\ell)). \end{gather*}
	By Lemma~\ref{simple-groups}, this map further factors through some $\Gal(L_iK(\zeta_\ell)/K(\zeta_\ell))$ with $1 \le i \le j-1$. 
	But this implies that $L_jK(\zeta_\ell) \subseteq L_iK(\zeta_\ell)$, contradicting claim~\ref{Lidisjoint} from the previous paragraph. 
	This concludes the proof.
\end{proof}

\begin{proposition}\label{congruences}
	Let $f_1,\ldots,f_s \in S_k(N)$ be normalized eigenforms such that for each $1\le i \le s$, $\rhobar_{f_i}(G_{\Q})$ contains a conjugate of $\SL_2(\F_\ell)$. 
	Then for any $\gamma \in \SL_2(\F_\ell)$, there is an element $\sigma \in \Gal(\Qbar/\Q(\zeta_{\ell}))$ such that $\rhobar_{f_i}(\sigma)$ is conjugate to $\gamma$ for each $1\le i \le s$.
\end{proposition}

\begin{proof}
	Say $r_{f_i} \cong {}^\tau r_{f_j}$ for some $i \ne j$ and $\tau \in \Gal(\Fbar_\ell/\F_\ell)$. 
	Then by Lemma~\ref{proj-im}, $\rhobar_{f_i}(\sigma)$ is conjugate to ${}^\tau\rhobar_{f_j}(\sigma)$ for any $\sigma \in \Gal(\Qbar/\Q(\zeta_{\ell}))$.
	If further $\rhobar_{f_j}(\sigma)$ is conjugate to $\gamma \in \SL_2(\F_\ell)$, then ${}^\tau\rhobar_{f_j}(\sigma)$ is conjugate to $\gamma$ as well. 
	We can thus assume that for each $1\le i \ne j \le s$, there is no $\tau \in \Gal(\Fbar_\ell/\F_\ell)$ such that $r_{f_i} \cong {}^\tau r_{f_j}$. 
	The result now follows from Lemma~\ref{galois-extension}.
\end{proof}

We can now prove Theorem~\ref{thm:lowweight}, using Propositions~\ref{big-image} and~\ref{congruences} as a base case for induction.


\begin{proof}[Proof of Theorem~\ref{thm:lowweight}]
	Choose a number field $E$ containing all Fourier coefficients of all the newforms in $S_{\ell-3}^{\mathrm{new}}(6)$. 
	Let $\lambda$ be the prime of $E$ induced by our fixed embedding $\iota_\ell \colon \Qbar \hookrightarrow \Qbar_\ell$. 
	Let $E_\lambda$ be the completion of $E$ at $\lambda$, let $\cO_\lambda$ be its ring of integers, and let $\F = \cO_\lambda/\lambda$ be the residue field. 
	Then for any newform $f \in S_{\ell-3}^{\mathrm{new}}(6)$, the Galois representations $\rho_f$ and $\rhobar_f$ of Theorem~\ref{gal-rep} can be defined over $\cO_\lambda$ and $\F$, respectively. 
	By Lemma~\ref{lem:from-lambda-to-ell}, there is an integer $r \ge m$ such that it suffices to show that there is a positive density set of primes $Q$ with $Q \equiv 1 \pmod{\ell^m}$ and $f\big|T(Q)\equiv f \pmod {\lambda^r}$ for all newforms $f \in S_{\ell-3}^{\mathrm{new}}(6)$.

By Propositions~\ref{big-image} and~\ref{congruences}, we can find an element $\sigma \in \Gal(\Qbar/\Q(\zeta_{\ell}))$ such that $\rhobar_f(\sigma)$ is conjugate to $\begin{pmatrix} 1 & 1 \\ -1 & 0 \end{pmatrix}$ for any newform $f \in S_{\ell-3}^{\mathrm{new}}(6)$. 
In particular, $\rho_f(\sigma)$ has characteristic polynomial congruent to $x^2 - x + 1 \pmod \lambda$. Enlarging $E$ if necessary, we can assume that $x^2 - x + 1$ factors in $\F$ and its roots are the two primitive 6-th roots of unity, which we denote by  $\xi$ and $\xi'$. Since $\xi$ and $\xi'$ are distinct, 
we can factor the characteristic polynomial of $\rho_f(\sigma)$ over $\mathcal{O}_\lambda$ by Hensel's lemma, and $\rho_f(\sigma)$ is conjugate to a diagonal matrix with entries $\alpha, \beta$ such that
$\alpha \equiv \xi \pmod{\lambda}$ and $\beta \equiv \xi' \pmod{\lambda}$. Letting $\xi$ and $\xi'$ again denote the primitive 6-th roots of unity in $\mathcal{O}_\lambda$, we write 
$\alpha=\xi \gamma$ and $\beta =\xi' \delta$ with $\gamma, \delta \equiv 1 \pmod \lambda$. Now $\rho_f(\sigma^{\ell^{r-1}})$ is conjugate to a diagonal matrix with entries
$\alpha^{\ell^{r-1}}=\xi^{\ell^{r-1}} \gamma^{\ell^{r-1}}$ and $\beta^{\ell^{r-1}} =\xi'^{\ell^{r-1}} \delta^{\ell^{r-1}}$. Observe that both $\gamma^{\ell^{r-1}}$ and $\delta^{\ell^{r-1}}$ are congruent to 1 modulo $\lambda^r$ (recall that $\lambda|\ell$), and that $\{\xi^{\ell^{r-1}},\xi'^{\ell^{r-1}}\}=\{\xi, \xi'\}$. Thus the characteristic polynomial of $\rho_f(\sigma^{\ell^{r-1}})$ is congruent to $(x-\xi)(x-\xi')=x^2-x+1$ modulo $\lambda^{r}$. Also, $\sigma^{\ell^{r-1}} \in \Gal(\Qbar/\Q(\zeta_{\ell^{m}}))$ since $\Gal(\Q(\zeta_{\ell^{m}})/\Q(\zeta_\ell))$ has order $\ell^{m-1}$ and $r \ge m$. Chebotarev's density theorem implies that for a positive density set of primes $Q$, $\Frob_Q$ is conjugate to $\sigma^{\ell^{r-1}}$. 
For such $Q$, we have $Q \equiv 1 \pmod{\ell^{m}}$ and for any newform $f \in S_{\ell-3}^{\mathrm{new}}(6)$, we have
\begin{gather*}f\big|T(Q)=(\operatorname{tr} \rho_f(\Frob_Q) )f \equiv f \pmod{\lambda^{r}}.\end{gather*} 
The theorem is now proven.
\end{proof}

\section{Congruences in arbitrary weight}
In this section, we use a different argument to prove two variants of Theorem \ref{thm:lowweight} in arbitrary integral weight $k\ge 2$ with additional hypotheses on the prime $\ell$. 
It will be convenient for us to fix a number field $E$ containing all Fourier coefficients of all newforms in $S_{k}^{\mathrm{new}}(6, \varepsilon_2, \varepsilon_3 )$. 
Let $\lambda$ be the prime of $E$ induced by our fixed embedding $\iota_\ell \colon \Qbar \hookrightarrow \Qbar_\ell$ and let $e$ be the ramification index.

We begin with an elementary lemma.
\begin{lemma} \label{elementary arithmetic}
	If  $a$ is an integer with  $2^a \equiv -2 \pmod{\ell}$, then 
	$2^{\ell^{m-1}(a-1)+1}\equiv -2\pmod {\ell^m}$ for any $m\geq 1$.
\end{lemma}
\begin{proof}
	We induct on $m$. Suppose that $2^b \equiv -1 \pmod{\ell^m}$, and write
	\begin{gather*}
	2^{\ell b}+1=(2^b+1)\( (2^b)^{\ell-1}-(2^b)^{\ell-2}+\cdots -2^b+1 \). 
	\end{gather*}
	Each summand in the second factor is 1 modulo $\ell$, and there are $\ell$ summands, so the second term is divisible by $\ell$.
\end{proof}

\begin{theorem} \label{weight k at 2}
	Suppose that $\ell\geq 5$ is prime and that there exists an integer $a$ for which $2^a \equiv -2 \pmod\ell$. Let $m$ be a natural number and let $\ep_2, \ep_3\in \{\pm1\}$.
	Then a positive proportion of primes $Q \equiv -2 \pmod{\ell^m}$ have the following property: for every $f \in S_{k}^{\mathrm{new}}(6, \varepsilon_2, \varepsilon_3 )$, we have $f\big|T(Q) \equiv -(-\ep_2)^a Q^{\frac{k-2}{2}} f \pmod{\ell^m}$.
\end{theorem}

\begin{proof} 
Let $E_\lambda$ be the completion of $E$ at $\lambda$, let $\cO_\lambda$ be its ring of integers, and let $\F = \cO_\lambda/\lambda$ be the residue field. 	
Then for any newform $f \in S_{k}^{\mathrm{new}}(6, \varepsilon_2, \varepsilon_3 )$, the Galois representations $\rho_f$ and $\rhobar_f$ of Theorem~\ref{gal-rep} can be defined over $\cO_\lambda$ and $\F$, respectively. 
By Lemma~\ref{lem:from-lambda-to-ell} and the  remark that follows it, there is an integer $r \ge m$ such that it suffices to show there is a positive proportion of primes $Q \equiv -2 \pmod{\ell^m}$ with the property that for every newform $f \in S_{k}^{\mathrm{new}}(6, \varepsilon_2, \varepsilon_3 )$, we have $f\big|T(Q) \equiv -(-\ep_2)^a Q^{\frac{k-2}{2}} f \pmod{\lambda^r}$.
By Lemma~\ref{elementary arithmetic}, $2^{\ell^{r-1}(a-1)+1}\equiv -2\pmod {\ell^r}$. 
Since $a$ and $\ell^{r-1}(a-1)+1$ have the same parity, we can replace $a$ with $\ell^{r-1}(a-1)+1$ and assume that $2^a \equiv -2 \pmod{\ell^r}$.

We have $(\rho_f|_{G_2})^{\mathrm{ss}} \cong \chi\psi \oplus \psi$ for the unramified character $\psi \colon G_2 \to \mathcal{O}_\lambda^{\times}$ with $\psi(\Frob_2)=\iota_\ell(a_2)$. By \cite[Theorem 3]{atkin-lehner}, $a_2=-\varepsilon_2 2^{\frac{k-2}{2}}$.  Let $K$ be the fixed field of the kernel of $\rho_f \mod \lambda^r$. By Chebotarev's density theorem,  a positive proportion of primes $Q$ have $\Frob_Q$  conjugate to $\Frob_2^a$ in $\Gal(K(\zeta_{\ell^r})/\Q)$. For such $Q$, we have 
\begin{gather*}
Q \equiv \chi(\Frob_Q) \equiv \chi(\Frob_2^a) \equiv 2^a \equiv -2 \pmod{\ell^r}.
\end{gather*}
 We also have
\begin{gather*}
\begin{aligned}
 a_Q = \operatorname{tr}\rho_f(\Frob_Q) \equiv \operatorname{tr} \rho_f(\Frob_2^a) &\equiv (-\varepsilon_2 2^{\frac{k-2}{2}})^a2^a+(-\varepsilon_2 2^{\frac{k-2}{2}})^a \\
&\equiv (-\varepsilon_2)^a2^{a\frac{k-2}{2}}(2^a+1) \equiv -(-\varepsilon_2)^a Q^{\frac{k-2}{2}} \mod \lambda^r.
\end{aligned}
\end{gather*}
The theorem is now proven.
\end{proof}

A similar argument establishes the following theorem, albeit with a stronger hypothesis and slightly weaker conclusion (due to the lack of an analogue of Lemma~\ref{elementary arithmetic}).

\begin{theorem} \label{weight k at 3}
	Suppose that $\ell\geq 5$ is prime, that  $m$ is a natural number, and that  there exists an integer $a$ such that $3^a \equiv -2 \pmod{\ell^m}$. Then 
	a positive proportion of primes $Q \equiv -2 \pmod{\ell^m}$ have the following property: for every newform $f \in S_{k}^{\mathrm{new}}(6, \varepsilon_2, \varepsilon_3)$, we have $f\big|T(Q) \equiv -(-\ep_3)^a Q^{\frac{k-2}{2}} f \pmod{\lambda^{em}}$.
\end{theorem}
\begin{proof} We proceed as in the last theorem with 2 replaced by 3 and $r=em$.
\end{proof}

\section{The application to partitions}
We use the results of the last two sections to prove Theorems~\ref{thm:Atkin1}--\ref{thm:Atkin3} from the Introduction.
We begin with  two lemmas.

\begin{lemma} \label{lem:shimura} 
Let   $\ell\geq 5$ be prime, and suppose that $f\in S_k\(1, \nu_\eta^r, \Z\)$ where $(r, 24)=1$ and $k\in \frac12\Z\setminus\Z$.
Suppose that $Q\ge 5$ is a prime and that $\lambda_Q$ is an integer with
\begin{gather*}
g\big |T(Q)\equiv \lambda_Q g\pmod\ell\ \ \text{ for all\ \ } 
g\in S_{2k-1}^{\operatorname{new}}\(6, -\pmfrac{8}{r}, -\pmfrac{12}{r}, \Z\).
\end{gather*}
Then 
\begin{gather*}
f\big | T(Q^2)\equiv \pmfrac{12}Q\lambda_Q f\pmod\ell.
\end{gather*}
\end{lemma}

\begin{proof}  
For each squarefree $t$ let
\begin{gather*}
F_t\in  S_{2k-1}^{\operatorname{new}}\(6, -\pmfrac{8}{r}, -\pmfrac{12}{r}, \Z\)
\end{gather*}
be the form with 
\begin{gather*}
\Sh_t f=F_t\otimes\pmfrac{12}\bullet.
\end{gather*}
For each $t$ we have
\begin{gather*}\(F_{t}\otimes \pmfrac{12}\bullet\)\big|T(Q)=\pmfrac{12}Q\(F_t\big|T(Q)\)\otimes \pmfrac{12}\bullet
\equiv \pmfrac{12}Q\lambda_Q F_{t}\otimes \pmfrac{12}\bullet\pmod\ell.
\end{gather*}
In other words, for each squarefree $t$ we have 
\begin{gather*}\Sh_t(f\big|T(Q^2))=\(\Sh_t f\)\big|T(Q)\equiv  \pmfrac{12}Q\lambda_Q \Sh_t f\pmod\ell.
\end{gather*}
The lemma now follows from \eqref{eq:shimcong}.
\end{proof}

The next lemma describes the consequence of finding a ``good" eigenvalue.
\begin{lemma}\label{lem:goodeigen}
Let   $\ell\geq 5$ be prime, and suppose that $f\in S_k\(1, \nu_\eta^r, \Z\)$ where $(r, 24)=1$ and $k\in \frac12\Z\setminus\Z$.
Suppose that   $Q\geq 5$ is prime and that  there exists $\alpha_Q\in \{\pm1\}$
with
\begin{gather*}
f\big|T(Q^2)\equiv \alpha_Q Q^{k-\frac32}f\pmod\ell.
\end{gather*}
Then we have 
\begin{gather*}
a(Q^2n)\equiv 0\pmod\ell\  \ \  \ \text{if}\ \  \ \pmfrac nQ=\alpha_Q\pmfrac{12}Q\pmfrac{-1}Q^{k-\frac12}.
\end{gather*}
\end{lemma}
\begin{proof} This follows from \eqref{eq:heckedef}.  For such $n$, the middle term in the definition of the Hecke operator 
cancels against the same term in $\alpha_Q Q^{k-\frac32}f$, and the third term  does not contribute.
\end{proof}

\begin{proof}[Proof of Theorem~\ref{thm:Atkin1}]
For $\ell\geq 13$, let  $f_\ell=\sum a(n)q^\frac n{24}\in S_\frac{\ell-2}2\(1, \nu_\eta^{-\ell}, \Z\)$ be the modular form  in \eqref{eq:fldef}.
By Theorem~\ref{thm:lowweight}, a positive proportion of primes $Q\equiv 1\pmod\ell$ have the property that for all 
$g\in S_{\ell-3}^{\operatorname{new}}\(6, -\pfrac{8}{-\ell}, -\pfrac{12} {-\ell}\)$ we have 
\begin{gather*}g\big|T(Q)\equiv g\pmod\ell.\end{gather*}
For such primes it follows from Lemma~\ref{lem:shimura}  that 
\begin{gather*}f_\ell\big|T(Q^2)\equiv \pmfrac{12}Q f_\ell\pmod\ell.\end{gather*}
Since $a(n)\equiv p\pfrac{\ell n+1}{24}\pmod\ell$, it follows from  Lemma~\ref{lem:goodeigen}
that 
\begin{gather*}p\pmfrac{Q^2\ell n+1}{24}\equiv 0\pmod\ell\qquad\text{if}\qquad \pmfrac nQ=\pmfrac{-1}Q^\frac{\ell-3}2.
\end{gather*}

\end{proof}

\begin{proof}[Proof of Theorem~\ref{thm:Atkin2}]
Suppose that $\ell\geq 13$ is a prime such that $2^a\equiv -2\pmod\ell$ for some $a$.  
 Theorem~\ref{weight k at 2}   guarantees that there exists $\beta\in \{\pm1\}$ and  a positive proportion of primes 
 $Q\equiv -2\pmod\ell$ such  that 
\begin{gather}\label{eq:gq_cong}
g\big|T(Q)\equiv \beta Q^\frac{\ell-5}2 g\pmod\ell\ \ \ \text{for all}\ \ \ g\in S_{\ell-3}^{\operatorname{new}}\(6, -\pfrac{8}{-\ell}, -\pfrac{12} {-\ell}, 
\Z\).
\end{gather}
It follows from Lemma~\ref{lem:shimura} that 
\begin{gather}\label{eq:flq_cong}
f_\ell\big|T(Q^2)\equiv \pmfrac{12}Q \beta Q^\frac{\ell-5}2 f_\ell\pmod\ell.
\end{gather}
From  Lemma~\ref{lem:goodeigen} we conclude that there are Type  I congruences for such primes $Q$.

To prove the existence of Type II congruences for $\ell\geq 5$ we argue in a similar  way, starting with the 
modular form 
$g_\ell=\sum b(n)q^\frac n{24}\in S_\frac{\ell^2-2}2\(1, \nu_\eta^{-1}, \Z\)$  defined in \eqref{eq:gldef}.
In this case we have 
\begin{gather*}b(n)\equiv p \( \mfrac{  n + 1}{24}\)\pmod\ell\ \ \ \ \text{when\ \  \ \ $\pmfrac{-n}\ell=-1$}.\end{gather*}

By  \eqref{eq:shimg} we have $\ep_2=-1$ in Theorem~\ref{weight k at 2}; 
we conclude using that result and  Lemma~\ref{lem:shimura} that for a positive proportion of primes $Q\equiv -2\pmod\ell$ we have 
\begin{gather*}
g_\ell\big|T(Q^2)\equiv -\pmfrac{12}Q Q^\frac{\ell^2-5}2 g_\ell\pmod\ell.
\end{gather*}
By Lemma~\ref{lem:goodeigen} we conclude that 
\begin{gather*}
b(Q^2n)\equiv 0\pmod\ell\  \ \  \ \text{if}\ \  \ \pmfrac nQ=-\pmfrac{-1}Q^{\frac{\ell^2-3}2}=-\pmfrac{-1}Q,
\end{gather*}
which gives a congruence of the form \eqref{eq:atkinII}.
\end{proof}

Finally, we turn to the proof of  Theorem~\ref{thm:Atkin3}.  Here the situation is complicated by the lack of an analogue of Lemma~\ref{elementary arithmetic} for the prime $3$;
this necessitates the added assumption that there are no congruences between newforms in the relevant spaces.
We first need a straighforward lemma.
\begin{lemma}\label{lem:mlambda}
Suppose that the space $S_{k}^{\operatorname{new}}\(6, \ep_2, \ep_3\)$  (where $k$ is even) is spanned by newforms $g_1, \dots, g_d$.
Let  $E$ be a number field containing the coefficients of $g_1, \dots g_d$, let $\cO$ be the ring of integers and let $\lambda$ be a prime of $E$ over the rational prime $\ell$.
Suppose that there is no
congruence $g_i\equiv g_j\pmod\lambda$ with $i\neq j$.
Then  if a non-zero modular form $F\in S_{k}^{\operatorname{new}}\(6, \ep_2, \ep_3, \cO\)$ is expressed as a linear combination
\begin{gather}\label{eq:lincomb}
F=\sum_{i=1}^d c_i g_i\qquad \text{with $c_i\in E$,}
\end{gather}
we have $\ord_\lambda(c_i)\geq 0$ for all $i$.
\end{lemma}
\begin{proof}[Proof of Lemma~\ref{lem:mlambda}]
  If the conclusion were false, then 
  clearing denominators in \eqref{eq:lincomb}  would show that the set $\{g_1, \dots, g_d\}$ is linearly dependent over 
$ \cO/\lambda$.
Let $j<d$ be the maximal index for which $\{g_1, \dots, g_j\}$ is linearly independent over $\cO/\lambda$.
Then there is a relation 
\begin{gather}\label{eq:minlincomb}
g_{j+1}\equiv \sum_{i=1}^j \alpha_i g_i\pmod\lambda.
\end{gather}
Write $g_i=\sum b_{i}(n)q^n$, and  assume without loss of generality that $\alpha_1\not\equiv 0\pmod\lambda$.
By   assumption we can find a prime $p\geq 5$ for which 
\begin{gather}
b_{j+1}(p)\not\equiv b_1(p)\pmod\lambda.
\end{gather}
Applying the Hecke operator $T(p)$ to \eqref{eq:minlincomb} gives
\begin{gather*}
b_{j+1}(p)\sum_{i=1}^j \alpha_ig_i\equiv \sum _{i=1}^j b_i(p) \alpha_i g_i.
\end{gather*}
Since $\alpha_1(b_{j+1}(p)-b_1(p))\not\equiv 0\pmod\lambda$, this gives a contradiction.
\end{proof}

\begin{proof}[Proof of Theorem~\ref{thm:Atkin3}]
Suppose that $\ell\geq 13$ is a prime such that $3^a\equiv -2\pmod\ell$ for some $a$.  
Applying Theorem~\ref{weight k at 3} with $m=1$ shows that 
there exists $\beta\in \{\pm1\}$ and  a positive proportion of primes 
 $Q\equiv -2\pmod\ell$ such  that for every newform
$g\in S_{\ell-3}^{\operatorname{new}}\(6, -\pfrac{8}{-\ell}, -\pfrac{12} {-\ell}\)$ we have 
\begin{gather*}g\big|T(Q)\equiv \beta Q^\frac{\ell-5}2 g\pmod\lambda.
\end{gather*}
By  Lemma~\ref{lem:mlambda} it follows that \eqref{eq:gq_cong} holds, and we argue as before to obtain the first conclusion of Theorem~\ref{thm:Atkin3}.
The second conclusion follows in similar fashion.
\end{proof}

\bibliographystyle{amsalpha}
\bibliography{atkin_cong}

\end{document}